\documentclass[12pt]{article}

\usepackage{graphicx}
\usepackage{mathptmx}      
\usepackage{latexsym}

\usepackage[utf8x]{inputenc}
\usepackage{ucs}
\usepackage{amsmath}
\usepackage{amsfonts}
\usepackage{amssymb}
\usepackage{amsthm}
\usepackage{enumerate}
\usepackage{geometry}
\usepackage{tikz}
\usetikzlibrary{decorations.pathmorphing}
\usetikzlibrary{fit,shapes}
\usepackage[round]{natbib}

\tikzstyle{var}=[ellipse,thick,draw=black,minimum size=1.2cm]

\theoremstyle{definition}
\newtheorem{proposition}{Proposition}
\newtheorem{lemma}[proposition]{Lemma}
\newtheorem{corollary}[proposition]{Corollary}
\newtheorem{example}[proposition]{Example}

\begin{document}

\title{Tutorial on Exact Belief Propagation in Bayesian Networks: from Messages to Algorithms.}

\author{Gregory Nuel~\footnote{Institute of Mathematics (INSMI), CNRS, Dept. of Applied Mathematics (MAP5), Paris Descartes University, France. {\tt gregory.nuel@parisdescartes.fr}}}

\date{January, 2012}



\maketitle

\begin{abstract}
In Bayesian networks, exact belief propagation is achieved through message passing algorithms. These algorithms (ex: inward and outward) provide only a recursive definition of the corresponding messages. In contrast, when working on hidden Markov models and variants, one classically first defines explicitly these messages (forward and backward quantities), and then derive all results and algorithms. In this paper, we generalize the hidden Markov model approach by introducing an explicit definition of the messages in Bayesian networks, from which we derive all the relevant properties and results including the recursive algorithms that allow to compute these messages. Two didactic examples (the precipitation hidden Markov model and the pedigree Bayesian network) are considered along the paper to illustrate the new formalism and standalone R source code is provided in the appendix.

\end{abstract}

\section{Introduction}

Probabilistic graphical models (PGMs) are powerful and versatile tools to study complex random systems with many variables \citep{cowell99,jensen07,koller09}. Causal PGMs are called Bayesian Networks (BNTs) and can be seen as a generalization of Markov models like Markov chains, Hidden Markov Models (HMMs), or Markov trees  \citep{smyth97}. For these models, exact inference usually involves the so-called \emph{forward} and \emph{backward} quantities which can be use to obtain marginal or conditional distributions. From the definition of these quantities one can derive recursive formulas that allow to obtain them through linear algorithms \citep{durbin98}.

In the case of BNTs, the same tasks is conducted through the exact Belief Propagation (BP) first introduced by \citet{pearl86,pearl88} for singly connected graphs and then generalized to multiply connected graphs by a serie of articles \citep{lauritzen88,shafer90,jensen90bis,jensen90}. Although many variants exist \citep{lepar98,schmidt98}, the principle of exact BP is always basically the same: 1) compute the so-called \emph{messages} through a recursive algorithm, 2) then combine them to obtain marginal or conditional distributions. As pointed out by \citet{smyth97}, these messages corresponds in fact exactly to the forward and backward quantities in the particular case of HMMs. However, there is a noticeable difference: in HMMs, messages are first defined explicitly and then used to derive results and algorithms, while with exact BP, messages are implicitly defined as the results of the recursion algorithms.

The objective of the present work is to push a step forward the parallel between HMMs and BNTs by introducing a new formalism where we first give an explicit sense to the messages from which all results, recursions, and algorithms can then be derived. 

The paper is organized as follows: in Section~\ref{section:hmm} we first consider a simple HMM example (the precipitation HMM) that will illustrate the message orientated approach of these models. In Section~\ref{section:bnt} we do some recalls on BNTs, the notion evidence, and junction tree. We also introduce a small but illustrative BNT example (the pedigree BNT). Finally in Section~\ref{section:results} we present our new results: the explicit definition of the message functions and how the classical results and algorithms derive from this definition. All results are illustrated both with the precipitation HMM and the pedigree BNT and standalone R source code is provided in the appendix. We end by discussing the possible advantages of this new approach.

\begin{table}[b]
\caption{Distribution of $Y_i$ conditionally to $S_i$ in the precipitation HMM.}
\label{tab:poisson}
$
{
\setlength\arraycolsep{5pt}
\begin{array}{cccccccccccc}
\hline
k & 0 & 1 & 2 & 3 & 4 & 5 & 6 & 7 & 8 & 9 & 10 \\
\hline
\mathbb{P}(Y_i=k | S_i=\text{L}) & 
.050& .149& .224& .224& .168& .101& .050& .022& .008& .003& .001\\
\mathbb{P}(Y_i=k | S_i=\text{H}) & 
.607& .303& .076& .013& .002& .000& .000& .000& .000& .000& .000\\
\hline
\end{array}
}
$
\end{table}

\section{Precipitation HMM}\label{section:hmm}

Let us assume that we observe daily the mm of precipitation at a given location. These measurements obviously depend on the atmospheric conditions. For simplification purpose, we consider only two possible atmospheric conditions: low pressure (denoted ${\tt L}$) and high pressure (denote ${\tt H}$). For $i=1,\ldots,n$, we denote by $Y_i$ the mm of precipitation observed at day $i$ and by $S_i$ the atmospheric conditions the same day, and we assume:
\begin{enumerate}[i)]
	\item $S_{1:n}=(S_i)_{i=1\ldots,n}$ is an homogeneous Markov chain starting with $S_1={\tt H}$, and with transition probabilities given by $\mathbb{P}(S_i={\tt L} | S_{i-1}={\tt H})=0.3$,  $\mathbb{P}(S_i={\tt H} | S_{i-1}={\tt L})=0.1$;
	\item $Y_{1;n}=(Y_i)_{i=1\ldots,n}$ is a independent sample of Poisson variables whose parameter only depends on $S_i$: $\mathbb{E}[Y_i | S_i={\tt L}]=3.0$ and $\mathbb{E}[Y_i | S_i={\tt H}]=0.5$ (see Tab.~\ref{tab:poisson}).
\end{enumerate}

We hence have:
\begin{equation}\label{eq:hmm_model}
\mathbb{P}(Y_{1:n},S_{1:n})=\mathbb{P}(S_1)\mathbb{P}(Y_1|S_1)\prod_{i=2}^n \mathbb{P}(S_i|S_{i-1})\mathbb{P}(Y_i|S_i)
\end{equation}

If $Y_{1:n}$ is observed while $S_{1:n}$ is not, this results in a typical HMM where there is a trend to have more precipitations in period of low atmospheric pressure. Our objective is to study $\mathbb{P}(S_{1:n} | Y_{1:n})$ the distribution of the unobserved phenomenon (the atmospheric pressure) conditionally to the observations (the mm of precipitation).

Following the classical approach to this problem \citep{durbin98}, we first introduce the so called \emph{forward} and \emph{backward} quantities, respectively defined for all $s\in \{{\tt L},{\tt H}\}$ and for $i=1\ldots n$ by:
\begin{equation}\label{eq:hmm_fb}
F_i(s)\stackrel{\text{def}}{=}\mathbb{P}(S_i=s,Y_{1:i}) \\
\quad\text{and}\quad
B_i(s)\stackrel{\text{def}}{=}\mathbb{P}(Y_{i+1:n} |S_i=s)
\end{equation}
with the convention that $B_n(s)=1$. The critical point is then just to prove the following proposition:

\begin{figure*}[t]
\begin{center}
\includegraphics[width=0.9\textwidth]{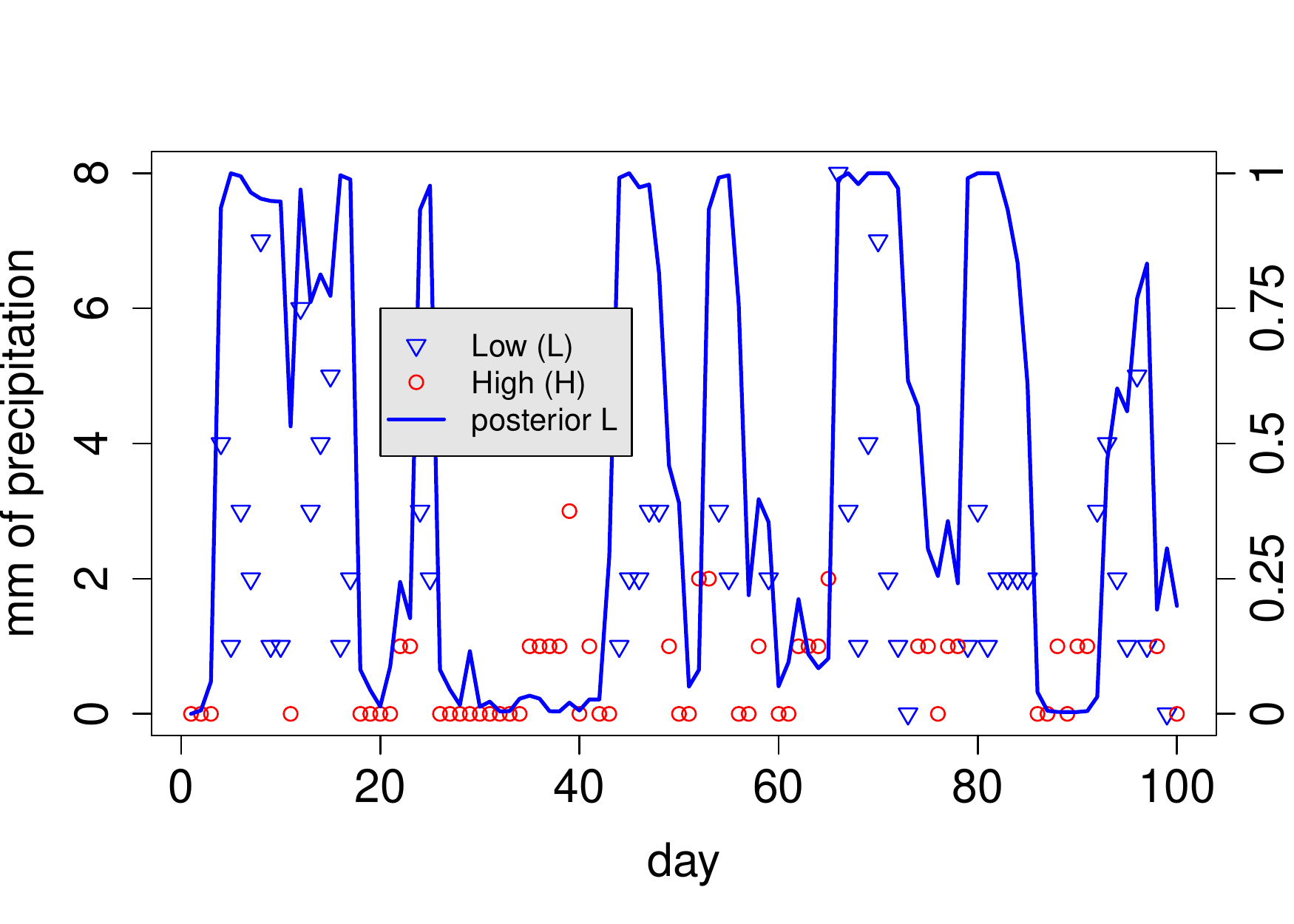}
\end{center}
\caption{Observation of the precipitation HMM over $n=100$ days. The (observed) mm of precipitation are given by the dots (scale on the left axis), the (hidden) status of states $S_i$ are given by the shape of the dots, and the posterior probability $\mathbb{P}(S_i={\tt L}|Y_{1:n})$ are given by the solid line (scale on the right axis).}
\label{fig:posterior_hmm}
\end{figure*}

\begin{proposition}
For all $i=1\ldots n$ and for all $r,s \in \{{\tt L},{\tt H}\}$ we have:
\begin{equation} \label{eq:hmm_marginal1}
\mathbb{P}(S_{i}=s,Y_{1:n})=F_{i}(s)B_{i}(s)
\end{equation}
and
\begin{equation} \label{eq:hmm_marginal2}
\mathbb{P}(S_{i-1}=r,S_{i}=s,Y_{1:n})=F_{i-1}(r)\pi(r,s)e_s(Y_i)B_{i}(s)
\end{equation}
where $\pi(r,s)\stackrel{\text{def}}{=}\mathbb{P}(S_{i}=s | S_{i-1}=r)$ and $e_s(k)\stackrel{\text{def}}{=}\mathbb{P}(Y_i=k | S_i=s)$.
\end{proposition}
\begin{proof}
	We prove only Eq.~(\ref{eq:hmm_marginal1}) since the argument is similar for Eq.~(\ref{eq:hmm_marginal2}). Thanks to Eq.~(\ref{eq:hmm_model}) we first observe that:
	$$
	\mathbb{P}(S_{1:n},Y_{1:n})=\mathbb{P}(S_{1:i},Y_{1:i})\mathbb{P}(S_{i+1:n},Y_{i+1:n} | S_i)
	$$
	from which a simple marginalization gives us:
	\begin{eqnarray*}
	\mathbb{P}(S_{i}=s,Y_{1:n})&=& \sum_{S_{1:i-1}} \sum_{S_{i+1:n}} \mathbb{P}(S_{1:i-1},S_{i}=s,Y_{1:i})
	\mathbb{P}(S_{i+1:n},Y_{i+1:n} |S_{i}=s) \\
	&=&\underbrace{\sum_{S_{1:i-1}}\mathbb{P}(S_{1:i-1},S_{i}=s,Y_{1:i})}_{F_i(s)}
		\underbrace{\sum_{S_{i+1:n}}\mathbb{P}(S_{i+1:n},Y_{i+1:n} |S_{i}=s)}_{B_i(s)}.
	\end{eqnarray*}
\end{proof}

From this proposition, we can easily establish all the classical results of HMM inference.

\begin{corollary}[forward and backward recursions]
The forward quantities can be recursively computed from $F_1(s)=\mathbf{1}_{s=\text{H}}e_s(Y_1)$ for all $i=2\ldots n$ with:
\begin{equation}\label{eq:hmm_forward_recursion}
F_{i}(s)=\sum_{r} F_{i-1}(r)\pi(r,s)e_s(Y_i)
\end{equation}
and similarly, the backward quantities can be recursively computed from $B_n(s)=1$ for all $i=n\ldots 2$:
\begin{equation}\label{eq:hmm_backward_recursion}
B_{i-1}(r)=\sum_{s} \pi(r,s)e_s(Y_{i})B_{i}(s).
\end{equation}
\end{corollary}
\begin{proof}
We only prove the forward recursion. We simply start from
$$\mathbb{P}(S_i=s,Y_{1:n})=\sum_{r} \mathbb{P}(S_{i-1}=r,S_i=s,Y_{1:n})$$
and apply Eq.~(\ref{eq:hmm_marginal1}) on the left-hand term, and Eq.~(\ref{eq:hmm_marginal2}) on the right-hand term to obtain:
$$
F_i(s)B_i(s)= \sum_{r} F_{i-1}(r)\pi(r,s)e_s(Y_i)B_{i}(s)
$$
which gives the forward recursion by simplifying by $B_{i}(s)$.
\end{proof}

We can see on Fig.~\ref{fig:posterior_hmm} and example of data produced by the model over $n=100$ days. The posterior probability $\mathbb{P}(S_i={\tt L}|Y_{1:n})$ is quite consistent with the (unobserved) reference values of $S_i$. 

\begin{table}[t]
\caption{Five samples drawn from $\mathbb{P}(S_{40:60} | Y_{1:n})$ using the data of Fig.~\ref{fig:posterior_hmm}. The reference value of $S_{i}$ and the posterior marginal distribution $\mathbb{P}(S_i={\tt L} | Y_{1:n})$ are also given for $i=40\ldots 60$.}
\label{tab:hmm_samples}
{\small
$
{
\setlength\arraycolsep{2pt}
\begin{array}{cccccccccccccccccccccc}
\hline
\text{day} & 40 & 41 & 42 & 43 & 44 &45&46&47&48&49&50&51&52&53&54&55&56&57&58&59&60\\
\hline
\text{reference} &  {\tt H} & {\tt H} & {\tt H} & {\tt H} & {\tt L} & {\tt L} & {\tt L} & {\tt L} & {\tt L} & {\tt H} & {\tt H} & {\tt H} & {\tt H} & {\tt H} & {\tt L} & {\tt L} & {\tt H} & {\tt H} & {\tt H} & 
{\tt L} & {\tt H} \\
\hline
\text{sample 1} & {\tt H} & {\tt H} & {\tt H} & {\tt H} & {\tt L} & {\tt L} & {\tt L} & {\tt H} & {\tt H} & {\tt H} & {\tt H} & {\tt H} & {\tt H} & {\tt L} & {\tt L} & {\tt L} & {\tt L} & {\tt L} & {\tt H} & {\tt L} & {\tt H} \\
\text{sample 2} & {\tt H} & {\tt H} & {\tt H} & {\tt H} & {\tt L} & {\tt L} & {\tt L} & {\tt L} & {\tt L} & {\tt H} & {\tt L} & {\tt H} & {\tt L} & {\tt L} & {\tt L} & {\tt L} & {\tt L} & {\tt H} & {\tt L} & {\tt L} & {\tt L} \\
\text{sample 3} & {\tt H} & {\tt H} & {\tt H} & {\tt L} & {\tt L} & {\tt L} & {\tt L} & {\tt L} & {\tt H} & {\tt H} & {\tt H} & {\tt H} & {\tt H} & {\tt L} & {\tt L} & {\tt L} & {\tt L} & {\tt H} & {\tt H} & {\tt H} & {\tt H} \\
\text{sample 4} & {\tt H} & {\tt H} & {\tt H} & {\tt H} & {\tt L} & {\tt L} & {\tt L} & {\tt L} & {\tt L} & {\tt L} & {\tt L} & {\tt H} & {\tt H} & {\tt L} & {\tt L} & {\tt L} & {\tt L} & {\tt H} & {\tt L} & {\tt L} & {\tt H} \\
\text{sample 5} & {\tt H} & {\tt H} & {\tt H} & {\tt H} & {\tt L} & {\tt L} & {\tt L} & {\tt L} & {\tt L} & {\tt H} & {\tt H} & {\tt H} & {\tt H} & {\tt L} & {\tt L} & {\tt L} & {\tt L} & {\tt H} & {\tt L} & {\tt H} & {\tt H} \\
\hline
\text{porterior ${\tt L}$} & .01 &.03 &.03 &.29 &.99 & 1.0&.97 &.98 &.82 &.46 &.39 &.05 &.08 &.93 &.99  &
1.0&.76 &.22 &.40 &.35 &.05 \\
\hline
\end{array}
}
$
}
\end{table}

\begin{corollary}[forward and backward sampling]
The distribution of $S_{1:n}$ conditionally to $Y_{1:n}$ is an heterogeneous Markov chain whose transitions are given by
\begin{equation}\label{eq:hmm_forward_sampling}
\mathbb{P}(S_{i}=s | S_{i-1}=r, Y_{1:n})
=\frac{\pi(r,s)e_s(Y_i)B_{i}(s)}{B_{i-1}(r)}
\end{equation}
in the forward direction, and by
\begin{equation}\label{eq:hmm_backward_sampling}
\mathbb{P}(S_{i-1}=r | S_{i}=s, Y_{1:n})
=\frac{F_{i-1}(r)\pi(r,s)e_s(Y_i)}{F_{i}(s)}
\end{equation}
in the backward direction.
\end{corollary}
\begin{proof}
We prove only the forward direction. We simply start from
$$
\mathbb{P}(S_{i}=s | S_{i-1}=r, Y_{1:n}) = \frac{\mathbb{P}(S_{i-1}=r,S_{i}=s, Y_{1:n})}{\mathbb{P}(S_{i-1}=r, Y_{1:n})}
$$
and use Eq.~(\ref{eq:hmm_marginal2}) on the numerator, and Eq.~(\ref{eq:hmm_marginal1}) on the denominator.
\end{proof}

For example, we can see on Tab.~\ref{tab:hmm_samples} some samples drawn from $\mathbb{P}(S_{1:n} | Y_{1:n})$ using the previous corollary.

\section{Recalls on Bayesian Networks}\label{section:bnt}

\subsection{Model}

Let $X_\mathcal{U}=(X_u)_{u \in \mathcal{U}}$, $\mathcal{U}=\{1,\ldots,p\}$ be a set of $p$ discrete\footnote{It is possible to consider continuous variables as well (or even a mixture of discrete and continuous variables) by replacing everywhere probabilities by densities, and sums by integrals. For the sake of simplicity, we here restrict ourselves to the pure discrete case.} random variables such as, for all $u\in \mathcal{U}$, $X_u \in \mathcal{D}_u \subset \mathbb{R}^{d_u}$ ($d_u \in \mathbb{N}^*$). Let $\mathcal{F} \subset \mathcal{U}\times \mathcal{U}$ such that $(\mathcal{U},\mathcal{F})$ define a directed acyclic graph (DAG) over $\mathcal{U}$. For all $v \in \mathcal{U}$, we define the \emph{parent set} of $v$ as $\text{pa}(v)\stackrel{\text{def}}{=}\{u\in \mathcal{U},(u,v)\in \mathcal{F}\}$. Then the distribution of $X_\mathcal{U} \in \mathcal{D}_\mathcal{U}$ is given by:
\begin{equation}\label{eq:model}
\mathbb{P}\left(X_\mathcal{U}\right) \stackrel{\text{def}}{=} \prod_{u \in \mathcal{U}}
\mathbb{P}\left( X_u | X_{\text{pa}(u)}\right).
\end{equation}
Note that Eq.~(\ref{eq:model}) defines a probability thanks to the acyclic property of graph $(\mathcal{U},\mathcal{F})$. Such a model is called a \emph{Bayesian network} (BNT) due to the fact the distribution of $X_\mathcal{U}$ is defined only through the conditional distributions $\mathbb{P}\left( X_u | X_{\text{pa}(u)}\right)$.

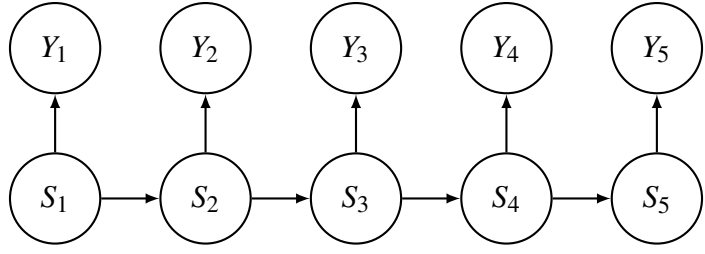
\begin{figure}[t]
\begin{center}
\begin{tikzpicture}[>=latex,text height=1.5ex,text depth=0.25ex,scale=1.0]
	\draw (0,0) node (S1) [var] {$S_1$};
	\draw (2,0) node (S2) [var] {$S_2$};
	\draw (4,0) node (S3) [var] {$S_3$};
	\draw (6,0) node (S4) [var] {$S_4$};
	\draw (8,0) node (S5) [var] {$S_5$};
	\draw (0,2) node (X1) [var] {$Y_1$};
	\draw (2,2) node (X2) [var] {$Y_2$};
	\draw (4,2) node (X3) [var] {$Y_3$};
	\draw (6,2) node (X4) [var] {$Y_4$};
	\draw (8,2) node (X5) [var] {$Y_5$};
	\path[->]
        (S1) edge[thick] (X1)
        (S2) edge[thick] (X2)
        (S3) edge[thick] (X3)
        (S4) edge[thick] (X4)
        (S5) edge[thick] (X5)
        (S1) edge[thick] (S2)
        (S2) edge[thick] (S3)
        (S3) edge[thick] (S4)
        (S4) edge[thick] (S5)
        ;
\end{tikzpicture}
\end{center}
\caption{DAG representing the precipitation HMM with $n=5$.}
\label{fig:dag_hmm}
\end{figure}

\begin{example}
In the particular case of the precipitation HMM over $n=100$ days we get the DAG of Fig.~\ref{fig:dag_hmm}.
If we denote the variable with $\mathcal{U}=\{-n,\ldots,-1\}\cup\{1,\ldots,n\}$ ($p=2n$), then for all $i=1\ldots n$ we have $X_{i}=S_i$ ($\mathcal{D}_{i}=\{{\tt L},{\tt H}\}$), $X_{-i}=Y_{i}$ ($\mathcal{D}_{-i}=\mathbb{N}$). We hence get the following parent sets: $\text{pa}(1)=\emptyset$, $\text{pa}(i)=\{i-1\}$ for $i=2\ldots n$, and $\text{pa}(-i)=\{i\}$ for $i=1\ldots n$. Note that replacing the generic variables by their values in Eq.~(\ref{eq:model}) immediately gives Eq.~(\ref{eq:hmm_model}).
\end{example}

\begin{example}
We can see on Fig.~\ref{fig:bnt_toy} a slightly more complex BNT which represents the parental relationships (a pedigree) of 10 individuals. This BNT includes a loop (consanguinity relationship between two cousins) but no orientated cycles.

The distribution of $X_{1:10}$ is hence given by
\begin{multline*}
\mathbb{P}(X_{1:10})=
\mathbb{P}(X_1)\mathbb{P}(X_2)\mathbb{P}(X_3|X_{1,2})\mathbb{P}(X_4|X_{1,2})\\ 
\mathbb{P}(X_5) \mathbb{P}(X_6) \mathbb{P}(X_7|X_{3,5})\mathbb{P}(X_8|X_{3,5})\mathbb{P}(X_9|X_{4,6})\mathbb{P}(X_{10}|X_{7,9}).
\end{multline*}

For all $i$, $X_i$ represents the genotype of individual $i$ at a given disease locus. We consider that there is only two alleles: the disease allele ${\tt D}$ and the non disease allele ${\tt d}$. $X_i$ hence takes its value in the following set of genotypes: $\{{\tt dd},{\tt dD},{\tt DD}\}$ (note that genotypes ${\tt dD}$ and ${\tt Dd}$ are indistinguishable).

For $i \in  \{1,2,5,6\}$ (the \emph{founders} set -- individuals with no parents), we assume a $20\%$ frequency for the disease allele in the general population and we get: $\mathbb{P}(X_i={\tt dd})=0.64$, $\mathbb{P}(X_i={\tt dD})=0.32$, and $\mathbb{P}(X_i={\tt DD})=0.04$. For any other individual $k$, we denote by $i$ and $j$ its two parents, and according to the Mendelian transmission of alleles we get the following conditional distribution:
$$
\begin{array}{cccccccccc}
\hline
X_i,X_j & {\tt dd},{\tt dd} & {\tt dd},{\tt dD} & {\tt dd},{\tt DD} & {\tt dD},{\tt dd} & {\tt dD},{\tt dD} &  {\tt dD},{\tt DD} & {\tt DD},{\tt dd} & {\tt DD},{\tt dD} & {\tt DD},{\tt DD}\\
\hline
\mathbb{P}(X_k={\tt dd}|X_i,X_j) & 1.00 & 0.50 & 0.00 & 0.50 & 0.25 & 0.00 & 0.00 & 0.00 & 0.00\\ 
\mathbb{P}(X_k={\tt dD}|X_i,X_j) & 0.00 & 0.50 & 1.00 & 0.50 & 0.50 & 0.50 & 1.00 & 0.50 & 0.00\\
\mathbb{P}(X_k={\tt DD}|X_i,X_j) & 0.00 & 0.00 & 0.00 & 0.00 & 0.25 & 0.50 & 0.00 & 0.50 & 1.00\\
\hline
\end{array}.
$$
\end{example}


\begin{figure}[t]
\begin{center}
\begin{tikzpicture}[>=latex,text height=1.5ex,text depth=0.25ex,scale=1.0]
	\draw (1,0) node (X1) [var] {$X_1$};
	\draw (5,0) node (X2) [var] {$X_2$};
	\draw (0,-2) node (X4) [var] {$X_4$};
	\draw (2,-2) node (X6) [var] {$X_6$};
	\draw (4,-2) node (X5) [var] {$X_5$};
	\draw (6,-2) node (X3) [var] {$X_3$};
	\draw (1,-4) node (X9) [var] {$X_9$};
	\draw (3,-4) node (X8) [var] {$X_8$};
	\draw (5,-4) node (X7) [var] {$X_7$};
	\draw (3,-6) node (X10) [var] {$X_{10}$};
	\path[->]
        (X1) edge[thick] (X4)
        (X1) edge[thick] (X3)
        (X2) edge[thick] (X4)
        (X2) edge[thick] (X3)
        (X3) edge[thick] (X7)
        (X3) edge[thick] (X8)
        (X4) edge[thick] (X9)
        (X6) edge[thick] (X9)
        (X5) edge[thick] (X7)
        (X5) edge[thick] (X8)
        (X9) edge[thick] (X10)
        (X7) edge[thick] (X10)
        ;
\end{tikzpicture}
\end{center}
\caption{Pedigree BNT of 10 individuals with a consanguinity loop between Individual~$7$ and Individual~$9$ (two cousins).}
\label{fig:bnt_toy}
\end{figure}
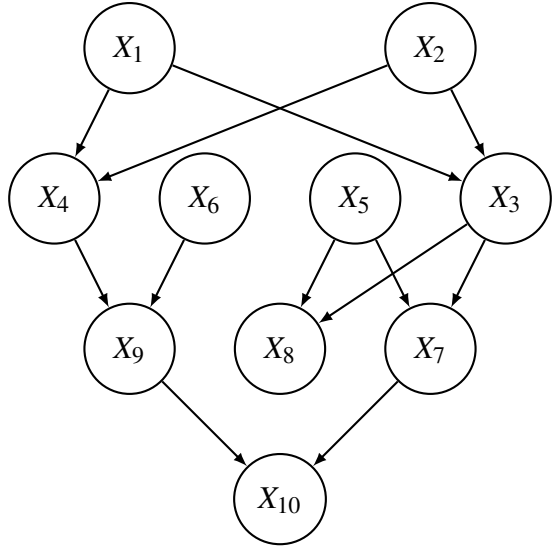

\subsection{Evidence}

We introduce the notion of \emph{evidence} by considering for all $u \in \mathcal{U}$ a subset $\mathcal{X}_u \subset \mathcal{D}_u$ of possible outcomes for $X_u$. For any $\mathcal{V} \subset \mathcal{U}$, we define $\mathcal{E}_\mathcal{V}\stackrel{\text{def}}{=}\{X_\mathcal{V} \in \mathcal{X}_\mathcal{V}\}$. Evidence is then defined as the event $\mathcal{E}\stackrel{\text{def}}{=}\mathcal{E}_\mathcal{U}=\{X_\mathcal{U} \in \mathcal{X}_\mathcal{U} \}$. Empty evidence (or no evidence) corresponds to the unconstrained case where $\mathcal{X}_u = \mathcal{D}_u$ for all $u \in \mathcal{U}$. Our aim is to study the conditional distribution
\begin{equation}
\mathbb{P}(X_\mathcal{U} | \mathcal{E})=\frac{\mathbb{P}(X_\mathcal{U} , \mathcal{E})}{\mathbb{P}( \mathcal{E})}.
\end{equation} 

\begin{example}
In the particular case of the precipitation HMM, we denote by $y_{1:n}$ the observed precipitations. We then have for all $i=1\ldots n$: $\mathcal{X}_i=\mathcal{D}_{i}=\{{\tt L},{\tt H}\}$ (no evidence), and $\mathcal{X}_{-i}=\{y_i\}$. We hence have $\mathcal{E}=\{ Y_{1:n}=y_{1:n} \}$ and $\mathbb{P}(X_\mathcal{U} | \mathcal{E})=\mathbb{P}(S_{1:n}| Y_{1:n}=y_{1:n})$.
\end{example}

\begin{example}
For the pedigree BNT, we assume that our disease locus is connected to a recessive disease. If a given individual $i$ is affected by the disease we have $X_i={\tt DD}$, if he is not affected we get $X_i \in \{ {\tt dd},{\tt dD} \}$. Assuming that individuals $8$, $9$ and $10$ are affected, that individual $7$ is not affected, and that we do not know the disease status of the remaining individuals, we get the following evidence: $\mathcal{E}=\{X_{1,3,5,6,9} \in \{{\tt dd},{\tt dD},{\tt DD} \}, X_7 \in \{{\tt dd},{\tt dD}\}, X_{2,4,8,10} = {\tt DD} \}$.

\end{example}

\subsection{Junction Tree}

We consider $C_\mathcal{I}=(C_i)_{i \in \mathcal{I}}$, $\mathcal{I}=\{1,\ldots,q\}$ a set of $q$ \emph{clusters} such as $C_i \subset \mathcal{U}$ for all $i \in \mathcal{I}$ and we assume the following three conditions:
\begin{enumerate}[JT1)]
\item {\bf Tree.} We have a tree structure on $C_\mathcal{I}$: for any $i,j \in \mathcal{I}$ it exists a unique connecting path, denoted $\text{path}(i,j)$, between $C_i$ and $C_j$.
\item {\bf Running intersection.}  For any $i,j \in \mathcal{I}$, $C_i \cap C_j \subset C_k$ for all $k \in \text{path}(i,j)$.
\item {\bf Covering.} For any $u \in \mathcal{U}$, it exists at least one $i \in \mathcal{I}$ such as the \emph{family set} $\text{fa}(u) \stackrel{\text{def}}{=} \text{pa}(u) \cup \{u\}\subset C_i$.
\end{enumerate}
Such a cluster tree is called a \emph{junction tree} (JT) associated to the BNT. Note that the tree composed by a single cluster $C_1=\mathcal{I}$ is always a junction tree, thus proving the existence of such object. However, finding a JT minimizing some criterion (typically the cardinal of the largest cluster) is known to be a NP-hard problem in general \citep{arnborg87}. Fortunately, it exists several heuristics that can build ``reasonable'', but possibly suboptimal, JTs \citep{jensen94,becker96,shoiket97}.

We assign for all $u \in \mathcal{U}$ a cluster $\text{cl}(u) \in \mathcal{I}$, such that $\text{fa}(u) \in \text{cl}(u)$. In the case that there are more than one cluster that fulfill this condition, we arbitrarily select one among them. Note that the condition (JT3) guarantees the existence of at least one possibility.

\begin{figure}[t]
{
\setlength\arraycolsep{1pt}
\begin{center}
\begin{tikzpicture}[>=latex,text height=1.5ex,text depth=0.25ex,scale=1.0]
	\draw (0.5,0) node (1) [var,label=below:$C_1$] {$\begin{array}{c}S_1^* \\ Y_1^*\end{array}$};
	\draw (3,0) node (2) [var,label=below:$C_2$] {$\begin{array}{cc}S_1 & S_2^* \\ \multicolumn{2}{c}{Y_2^*}\end{array}$};
	\draw (6,0) node (3) [var,label=below:$C_3$] {$\begin{array}{cc}S_2 & S_3^* \\ \multicolumn{2}{c}{Y_3^*}\end{array}$};
	\draw (9,0) node (4) [var,label=below:$C_4$] {$\begin{array}{cc}S_3 & S_4^* \\ \multicolumn{2}{c}{Y_4^*}\end{array}$};
	\draw (12,0) node (5) [var,label=below:$C_5$] {$\begin{array}{cc}S_4 & S_5^* \\ \multicolumn{2}{c}{Y_5^*}\end{array}$};;
	\path[-]
        (1) edge[thick] (2)
        (2) edge[thick] (3)
        (3) edge[thick] (4)
        (4) edge[thick] (5)
		;
\end{tikzpicture}
\end{center}
}
\caption{JT for the precipitation HMM with $n=5$. The star $^*$ indicates the cluster to which is associated each variable.}
\label{fig:jt_hmm}
\end{figure}
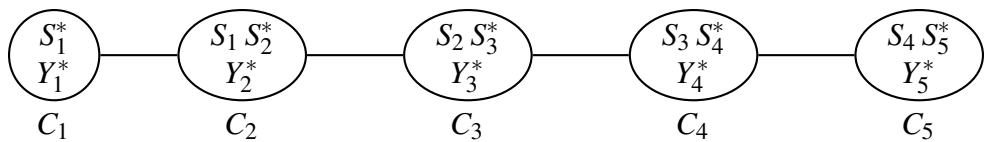

\begin{example}
In the particular case of the precipitation HMM we can build the simple JT  which is a chained sequence of $n$ clusters: $C_1=\{1,-1\}$ and $C_i=\{i-1,i,-i\}$ for $i=2\ldots n$. In order to improve readability from now on we will use the original name of the variables rather than its index $u$ (ex: $S_4$ instead of $4$, $X_2$ instead of $-2$) whenever the notation is not ambiguous. We can therefore write $C_1=\{S_1,Y_1\}$ and $C_i=\{S_{i-1},S_i,Y_i\}$ for $i=2\ldots n$
(see of Fig.~\ref{fig:jt_hmm} for an example with $n=5$). The resulting structure obviously fulfills the three JT conditions. For $i=1\ldots n$, variables $Y_i$ and $S_i$ are assigned to cluster $C_i$. 
\end{example}

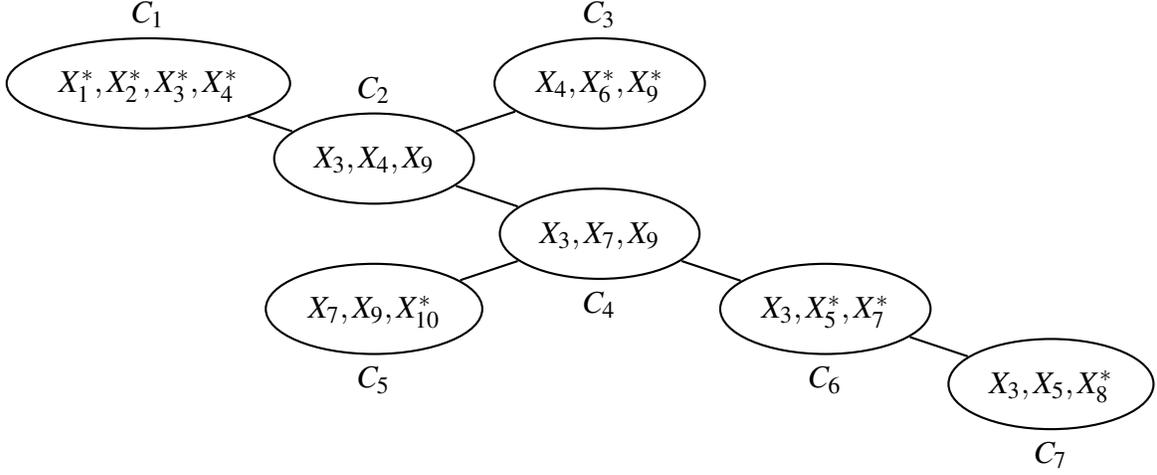
\begin{figure}[t]
\begin{center}
\begin{tikzpicture}[>=latex,text height=1.5ex,text depth=0.25ex,scale=1.0]
	\draw (1,-2) node (C1) [var,label=above:$C_1$] {$X_1^*,X_2^*,X_3^*,X_4^*$};
	\draw (4,-3) node (C2) [var,label=above:$C_2$] {$X_3,X_4,X_9$};
	\draw (7,-2) node (C3) [var,label=above:$C_3$] {$X_4,X_6^*,X_9^*$};
	\draw (7,-4) node (C4) [var,label=below:$C_4$] {$X_3,X_7,X_9$};
	\draw (4,-5) node (C5) [var,label=below:$C_5$] {$X_7,X_9,X_{10}^*$};
	\draw (10,-5) node (C6) [var,label=below:$C_6$] {$X_3,X_5^*,X_7^*$};
	\draw (13,-6) node (C7) [var,label=below:$C_7$] {$X_3,X_5,X_8^*$};
	\path[-]
        (C1) edge[thick] (C2)
        (C2) edge[thick] (C3)
        (C2) edge[thick] (C4)
        (C4) edge[thick] (C5)
        (C4) edge[thick] (C6)
        (C6) edge[thick] (C7)
		;
\end{tikzpicture}
\end{center}
\caption{JT for the pedigree BNT. The star $^*$ indicates the cluster to which is associated each variable.}
\label{fig:jt_toy}
\end{figure}

\begin{example}
We can see in Fig.~\ref{fig:jt_toy}, a JT associated to the pedigree BNT of Fig.~\ref{fig:bnt_toy}. Conditions JT1 (tree) and JT3 (covering) are clearly respected. This is also true of JT2 (running intersection) even if it is less obvious. For an illustrative purpose, let us verify JT2 in two particular cases: 1) $C_1 \cap C_7=\{X_3\}$ which means that $C_2$ and $C_4$ must contain $X_3$; 2) $C_3 \cap C_5=\{X_9\}$ which means that $C_2$ and $C_4$ must also contain $X_9$.
\end{example}

\section{Results}\label{section:results}

\subsection{Messages}


For any edge $i - j$ of the JT, we define the following two sets:
$S_{i,j}=S_{j,i}\stackrel{\text{def}}{=}C_i \cap C_j$ the \emph{separator set}; $U_{i\rightarrow j}\stackrel{\text{def}}{=}\{u \in \mathcal{U}, i \in \text{path}(\text{cl}(u),j)\}$ the \emph{upstream set} ($U_{i\rightarrow j} \cup U_{j\rightarrow i}$ is a partition of $\mathcal{U}$). We then define the \emph{message function} $M_{i\rightarrow j}$ for all $X_{S_{i,j}} \in \mathcal{D}_{S_{i,j}}$ by:
\begin{equation}\label{eq:messagedef0}
	M_{i \rightarrow j}\left(X_{S_{i,j}}\right)\stackrel{\text{def}}{=}
		\mathbf{1}_{\mathcal{E}_{L_{i\rightarrow j}}}%
	\mathbb{P}\left(
	X_{L_{i\rightarrow j}},
	\mathcal{E}_{V_{i\rightarrow j}}
	\left| X_{L_{j\rightarrow i}} \right.
	\right)
\end{equation}
with the convention that $\mathbf{1}_{\mathcal{E}_\emptyset}=1$ and with $L_{i\rightarrow j}\stackrel{\text{def}}{=}U_{i\rightarrow j}\cap S_{i,j}$ ($L_{i\rightarrow j} \cup L_{j\rightarrow i}$ is a partition of $S_{i,j}$), $V_{i\rightarrow j}\stackrel{\text{def}}{=}U_{i\rightarrow j}\setminus S_{i,j}$ ($V_{i\rightarrow j} \cup V_{j\rightarrow i}$ is a partition of $\mathcal{U}\setminus S_{i,j}$).


\begin{example}
In the particular case of the precipitation HMM, for all $i=1\ldots n-1$ we have $S_{i,i+1}=\{S_i\}$, $U_{i\rightarrow i+1}=\{S_{1:i},Y_{1:i}\}$, and $U_{i+1\rightarrow i}=\{S_{i+1:n},Y_{i+1:n}\}$. We hence have
$M_{i \rightarrow i+1}\left(S_i\right)
	=\mathbb{P}(S_i,Y_{1:i}=y_{1:i})$
and
$M_{i+1 \rightarrow i}\left(S_i\right)
	=\mathbb{P}(Y_{1:i}=y_{1:i}|S_i)
$.
We recognize the forward and backward quantities of Eq.~(\ref{eq:hmm_fb}).
\end{example}

\begin{example}
For the pedigree BNT and the JT of  Fig.~\ref{fig:jt_toy} we obtain the following messages:
\begin{itemize}
\item $M_{1 \rightarrow 2}(X_{3:4})=\mathbf{1}_{\mathcal{E}_{3:4}} \mathbb{P}(X_{3:4}, \mathcal{E}_{1:2})$,
$M_{2 \rightarrow 1}(X_{3:4})= \mathbb{P}(\mathcal{E}_{5:10}| X_{3:4})$;
\item $M_{2 \rightarrow 3}(X_{4},X_9)=\mathbf{1}_{\mathcal{E}_{4}}\mathbb{P}(X_4,\mathcal{E}_{1:3,5,7:8,10}| X_{9})$, $M_{3 \rightarrow 2}(X_{4},X_9)=\mathbf{1}_{\mathcal{E}_{9}} \mathbb{P}(X_9 , \mathcal{E}_{6}| X_{4})$;
\item $M_{2 \rightarrow 4}(X_{3},X_9)=\mathbf{1}_{\mathcal{E}_{3,9}} \mathbb{P}(X_{3,9},\mathcal{E}_{1:2,4,6})$,
$M_{4 \rightarrow 2}(X_{3},X_9)=\mathbb{P}(\mathcal{E}_{5,7:8,10}| X_{3,9})$;
\item $M_{4 \rightarrow 5}(X_{7},X_9)=\mathbf{1}_{\mathcal{E}_{7,9}} \mathbb{P}(X_{7,9},\mathcal{E}_{1:6,8})$,
$M_{5 \rightarrow 4}(X_{7},X_9)=\mathbb{P}(\mathcal{E}_{10}| X_{7,9})$;
\item $M_{4 \rightarrow 6}(X_{3},X_7)=\mathbf{1}_{\mathcal{E}_{3}} \mathbb{P}(X_{3},\mathcal{E}_{1:2,4,6,9:10} | X_7)$,
$M_{6 \rightarrow 4}(X_{3},X_7)=\mathbf{1}_{\mathcal{E}_{7}}\mathbb{P}(X_7,\mathcal{E}_{5,8}| X_{3})$;
\item $M_{6 \rightarrow 7}(X_{3},X_5)=\mathbf{1}_{\mathcal{E}_{3,5}} \mathbb{P}(X_{3,5},\mathcal{E}_{1:2,4,6:7,9:10})$,
$M_{7 \rightarrow 6}(X_{3},X_5)=\mathbb{P}(\mathcal{E}_{8}| X_{3,5})$.
\end{itemize}
\end{example}

\begin{lemma}\label{lemma:messages}
For all $u \in \mathcal{U}$, we introduce the \emph{potential} $K_u\left(X_{\text{fa}(u)}\right)\stackrel{\text{def}}{=}\mathbf{1}_{\mathcal{E}_u}\mathbb{P}\left( X_u | X_{\text{pa}(u)}\right)$ and get:
\begin{equation}\label{eq:messagedef}
	M_{i \rightarrow j}\left(X_{S_{i,j}}\right)=
	\sum_{X_{V_{i \rightarrow j}}}	
	\prod_{u \in U_{i\rightarrow j}} K_u\left(X_{\text{fa}(u)}\right).
\end{equation}
\end{lemma}
\begin{proof}
From the definition of the potential $K_u$, it is first clear that
$$
\prod_{u \in U_{i\rightarrow j}} K_u\left(X_{\text{fa}(u)}\right)=
\mathbf{1}_{\mathcal{E}_{U_{i \rightarrow j}}}\mathbb{P}(X_{U_{i \rightarrow j}}| X_{\text{fa}(U_{i\rightarrow j}) \setminus U_{i\rightarrow j}}).
$$
Moreover if $u \in \text{fa}(U_{i\rightarrow j})\setminus U_{i\rightarrow j}$, the covering property ensure that $u$ appears at least  once the upstream side of $i \rightarrow j$. Moreover, since $U_{i\rightarrow j} \cup U_{j\rightarrow i}=\mathcal{U}$ is a partition, $u$ also appears on the downstream side of $i \rightarrow j$. The running intersection property hence proves that $u \in C_i \cap C_j=S_{i,j}$. Since $\text{fa}(U_{i\rightarrow j}) \setminus U_{i\rightarrow j} \subset S_{i,j}$ we therefore can write:
$$
\mathbf{1}_{\mathcal{E}_{U_{i \rightarrow j}}}\mathbb{P}(X_{U_{i \rightarrow j}}| X_{\text{fa}(U_{i\rightarrow j}) \setminus U_{i\rightarrow j}})=
\mathbf{1}_{\mathcal{E}_{L_{i \rightarrow j}}}\mathbb{P}(X_{L_{i \rightarrow j}},
X_{V_{i \rightarrow j}} | X_{L_{j \rightarrow i}})
$$
and the summation over $X_{V_{i \rightarrow j}}$ immediately proves the lemma.
\end{proof}	

Although it is not proved in the same way, one should not that this lemma corresponds exactly to Theorem~10.3 page~354 in \citet{koller09}.

\begin{example}
For the pedigree BNT, we obtain the following potentials:

\vspace{0.5em}

\noindent \hspace{-1.5em} {\setlength{\arraycolsep}{0pt}
\begin{tabular}{ll}
\begin{minipage}{0.46\textwidth}
\begin{itemize}
\item $K_1(X_1)=\mathbf{1}_{\mathcal{E}_1} \mathbb{P}(X_1)$;
\item $K_2(X_2)=\mathbf{1}_{\mathcal{E}_2} \mathbb{P}(X_2)$;
\item $K_3(X_1,X_2,X_3)=\mathbf{1}_{\mathcal{E}_3} \mathbb{P}(X_3|X_1,X_2)$;
\item $K_4(X_1,X_2,X_4)=\mathbf{1}_{\mathcal{E}_4} \mathbb{P}(X_4|X_1,X_2)$;
\item $K_5(X_5)=\mathbf{1}_{\mathcal{E}_5} \mathbb{P}(X_5)$;
\end{itemize}
\end{minipage}
&
\begin{minipage}{0.55\textwidth}
\begin{itemize}
\item $K_6(X_6)=\mathbf{1}_{\mathcal{E}_6} \mathbb{P}(X_6)$;
\item $K_7(X_3,X_5,X_7)=\mathbf{1}_{\mathcal{E}_7} \mathbb{P}(X_7|X_3,X_5)$;
\item $K_8(X_3,X_5,X_8)=\mathbf{1}_{\mathcal{E}_8} \mathbb{P}(X_8|X_3,X_5)$;
\item $K_9(X_4,X_6,X_9)=\mathbf{1}_{\mathcal{E}_9} \mathbb{P}(X_9|X_4,X_6)$;
\item $K_{10}(X_7,X_9,X_{10})=\mathbf{1}_{\mathcal{E}_{10}} \mathbb{P}(X_{10}|X_7,X_9)$.
\end{itemize}
\end{minipage}
\end{tabular}
}
\end{example}

\subsection{Marginal distributions}

\begin{proposition}\label{prop:marginal1}
For any edge $i-j$ of the JT, and for all $X_{S_{i,j}} \in \mathcal{D}_{S_{i,j}}$ we have:
\begin{equation}\label{eq:marginal1}
	\mathbb{P}\left(X_{S_{i,j}},\mathcal{E}\right)=
	M_{i\rightarrow j}\left(X_{S_{i,j}}\right)
	M_{j\rightarrow i}\left(X_{S_{i,j}}\right).
\end{equation}
\end{proposition}
\begin{proof}
Starting from
$$
\mathbb{P}\left(X_{S_{i,j}},\mathcal{E}\right)=
\sum_{X_{\mathcal{U}\setminus S_{i,j}}}
\prod_{u \in \mathcal{U}} K_u\left(X_{\text{fa}(u)}\right)
$$
with use the fact that $V_{i\rightarrow j} \cup V_{j\rightarrow i}$ is a partition of $\mathcal{U}\setminus S_{i,j}$ and that  $U_{i\rightarrow j} \cup U_{j\rightarrow i}$ is a partition of $\mathcal{U}$ to write:
$$
\mathbb{P}\left(X_{S_{i,j}},\mathcal{E}\right)=
\sum_{X_{V_{i \rightarrow j}}}\sum_{X_{V_{j \rightarrow i}}}
\prod_{u \in U_{i \rightarrow j}} K_u\left(X_{\text{fa}(u)}\right)
\prod_{u \in U_{j \rightarrow i}} K_u\left(X_{\text{fa}(u)}\right)
$$
and since for all $u \in U_{i \rightarrow j}$ it is clear that $K_u\left(X_{\text{fa}(u)}\right)$ does not depend on $X_{V_{j \rightarrow i}}$ we finally obtain:
$$
\mathbb{P}\left(X_{S_{i,j}},\mathcal{E}\right)=
\underbrace{\sum_{X_{V_{i \rightarrow j}}}
\prod_{u \in U_{i \rightarrow j}} K_u\left(X_{\text{fa}(u)}\right)}_{M_{i\rightarrow j}\left(X_{S_{i,j}}\right)}
\underbrace{\sum_{X_{V_{j \rightarrow i}}}
\prod_{u \in U_{j \rightarrow i}} K_u\left(X_{\text{fa}(u)}\right)}_{M_{j\rightarrow i}\left(X_{S_{i,j}}\right)}
$$
which achieves the proof.
\end{proof}

\begin{example}
For the precipitation HMM, Proposition~\ref{prop:marginal1} gives for all $i=1\ldots n-1$:
$$
\mathbb{P}(S_i,Y_{1:n}=y_{1:n})=M_{i\rightarrow i+1}(S_i)M_{i+1\rightarrow i}(S_i)=F_i(S_i)B_i(S_i)
$$
which is exactly Eq~(\ref{eq:hmm_marginal1}).
\end{example}

\begin{proposition}\label{prop:marginal2}
For any $j \in \mathcal{J}$ and for all $X_{C_j} \in \mathcal{D}_{C_j}$ we have:
\begin{equation}\label{eq:marginal2}
	\mathbb{P}\left(X_{C_j},\mathcal{E}\right)=
	\Phi_j\left(X_{C_j}\right)
	\prod_{i \in \text{n}(j)} M_{i\rightarrow j}\left(X_{S_{i,j}}\right)
\end{equation}
where $\text{n}(j)\stackrel{\text{def}}{=}\{i, \text{$i-j$ is an edge of the JT}\}$ denotes the \emph{neighbor set} of $j$, and where
$\Phi_j\left(X_{C_j}\right)\stackrel{\text{def}}{=} \prod_{u \in C_j^* } K_u\left(X_{\text{fa}(u)}\right)$,  with $C_j^*\stackrel{\text{def}}{=}\{u \in C_j, \text{cl}(u)=j\}$, is the potential of $C_j$.
\end{proposition}
\begin{proof}
The proof is very similar to the one of Proposition~\ref{prop:marginal1}. The key is here to realize that: 1) $\cup_{i\in\text{n}(j)} V_{i \rightarrow j}$ is a partition of $\mathcal{U} \setminus C_j$; 2)
$C_j^* \cup_{i\in\text{n}(j)} U_{i \rightarrow j}$ is a partition of $\mathcal{U}$.
\end{proof}


\begin{example}
In the particular case of the precipitation HMM, for $i=2\ldots n-1$ we have
$C_{i}=\{S_{i-1},S_i,Y_i\}$, $C_i^*=\{S_{i},Y_i\}$, $\text{n}(i)=\{i-1,i+1\}$, and hence $\mathbb{P}(S_{i-1},S_{i},Y_{1:n}=y_{1:n})=\mathbb{P}(S_i|S_{i-1})\mathbb{P}(Y_i=y_i|S_i)M_{i-1\rightarrow i}(S_{i-1})M_{i+1\rightarrow i}(S_{i})$, which corresponds to Eq~(\ref{eq:hmm_marginal2}).
\end{example}


\begin{example}
For the pedigree BNT, the marginal distributions of all clusters are the following:
\begin{itemize}
\item $\mathbb{P}(X_1,X_2,X_3,X_4,\mathcal{E})=K_1(X_1)K_2(X_2)K_3(X_1,X_2,X_3)K_4(X_1,X_2,X_4)M_{2\rightarrow 1}(X_3,X_4)$;
\item $\mathbb{P}(X_3,X_4,X_9,\mathcal{E})=M_{1\rightarrow 2}(X_3,X_4)M_{3\rightarrow 2}(X_4,X_9)M_{4\rightarrow 2}(X_3,X_9)$;
\item $\mathbb{P}(X_4,X_6,X_9,\mathcal{E})=K_6(X_6)K_9(X_4,X_6,X_9)M_{2\rightarrow 3}(X_4,X_9)$;
\item $\mathbb{P}(X_3,X_7,X_9,\mathcal{E})=M_{2\rightarrow 4}(X_3,X_9)M_{5\rightarrow 4}(X_7,X_9)M_{6\rightarrow 4}(X_3,X_7)$;
\item $\mathbb{P}(X_7,X_9,X_{10},\mathcal{E})=K_{10}(X_7,X_9,X_{10})M_{4\rightarrow 5}(X_7,X_9)$;
\item $\mathbb{P}(X_3,X_5,X_7,\mathcal{E})=K_5(X_5)K_7(X_3,X_5,X_7)M_{4\rightarrow 6}(X_3,X_7)M_{7\rightarrow 6}(X_3,X_5)$;
\item $\mathbb{P}(X_3,X_5,X_8,\mathcal{E})=K_8(X_3,X_5,X_8)M_{6\rightarrow 7}(X_3,X_5)$.
\end{itemize}

Using the messages computed in Table~\ref{table:ped_messages} (see next section for more details on this computation), we get:
\begin{itemize}
\item $\mathbb{P}(\mathcal{E})=\sum_{X_3} M_{1\rightarrow 2}(X_3,{\tt DD})M_{2\rightarrow 1}(X_3,{\tt DD})
=0.0000480+0.0001152=0.0001632$;
\item $\mathbb{P}(X_1={\tt dD} | \mathcal{E})=0.7647$, and $\mathbb{P}(X_1={\tt DD} | \mathcal{E})=0.2353$;
\item $\mathbb{P}(X_2={\tt DD} | \mathcal{E})=1.0000$;
\item $\mathbb{P}(X_3={\tt dD} | \mathcal{E})=0.2941$, and $\mathbb{P}(X_3={\tt DD} | \mathcal{E})=0.7059$;
\item $\mathbb{P}(X_4={\tt DD} | \mathcal{E})=1.0000$;
\item $\mathbb{P}(X_5={\tt dD} | \mathcal{E})=0.9412$, and $\mathbb{P}(X_5={\tt DD} | \mathcal{E})=0.0588$;
\item $\mathbb{P}(X_6={\tt dd} | \mathcal{E})=0.5333$, $\mathbb{P}(X_6={\tt dD} | \mathcal{E})=0.4000$, and $\mathbb{P}(X_6={\tt DD} | \mathcal{E})=0.0667$;
\item $\mathbb{P}(X_7={\tt dD} | \mathcal{E})=1.0000$;
\item $\mathbb{P}(X_8={\tt DD} | \mathcal{E})=1.0000$;
\item $\mathbb{P}(X_1={\tt dD} | \mathcal{E})=0.6778$, and $\mathbb{P}(X_1={\tt DD} | \mathcal{E})=0.3333$;
\item $\mathbb{P}(X_{10}={\tt DD} | \mathcal{E})=1.0000$.
\end{itemize}
One should note that these marginal distributions only describe roughly the distribution $\mathbb{P}(X_\mathcal{U}|\mathcal{E})$. For example, if we consider the joint distribution of $(X_3,X_5)$ (obtained by the product of messages $M_{6\rightarrow 7}$ and $M_{7\rightarrow 6}$) we get: $\mathbb{P}(X_3={\tt dD},X_{5}={\tt dD} | \mathcal{E})=0.2353$, $\mathbb{P}(X_3={\tt dD},X_{5}={\tt DD} | \mathcal{E})=0.0588$, $\mathbb{P}(X_3={\tt DD},X_{5}={\tt dD} | \mathcal{E})=0.7059$, and $\mathbb{P}(X_3={\tt DD},X_{5}={\tt DD} | \mathcal{E})=0.000$ while (for example) $\mathbb{P}(X_3={\tt DD} | \mathcal{E})\times \mathbb{P}(X_5={\tt DD} | \mathcal{E})=0.0415 \neq 0.000$.

\end{example}

\subsection{Recursions}

\begin{corollary}
For all $j-k$ edge of the JT, for all $X_{S_{j,k}} \in \mathcal{D}_{S_{j,k}}$ we have:
\begin{equation}\label{eq:recursion}
M_{j \rightarrow k}\left(X_{S_{j,k}}\right)=
\sum_{X_{C_j\setminus S_{j,k}}}
\Phi_j\left( X_{C_j} \right)
\prod_{i \in \text{n}(j),i \neq k} M_{i\rightarrow j}\left(X_{S_{i,j}}\right).
\end{equation}
\end{corollary}
\begin{proof}
Start with
$$\mathbb{P}\left(X_{S_{j,k}},\mathcal{E}\right)=\sum_{X_{C_j\setminus S_{j,k}}}
\mathbb{P}\left(X_{C_j},\mathcal{E}\right)$$ and apply 
Eq.~(\ref{eq:marginal1}) to the left-hand and Eq.~(\ref{eq:marginal2}) to the right-hand.
\end{proof}

\begin{example}
In the particular case of the precipitation HMM, we get:
\begin{itemize}
\item for all $i=2\ldots n-1$, $M_{i \rightarrow i+1}\left(S_i\right)=
\sum_{S_{i-1}}\mathbb{P}(S_{i}|S_{i-1})\mathbb{P}(Y_i=y_i|S_i)
 M_{i-1\rightarrow i}\left(S_{i-1}\right)$ which is exactly the forward recursion of Eq.~(\ref{eq:hmm_forward_recursion});
\item for all $i=1\ldots n-2$, $M_{i+1 \rightarrow i}\left(S_i\right)=
\sum_{S_{i+1}}\mathbb{P}(S_{i+1}|S_{i})\mathbb{P}(Y_{i+1}=y_{i+1}|S_{i+1})
 M_{i\rightarrow i-1}\left(S_{i}\right)$ which is exactly the forward recursion of Eq.~(\ref{eq:hmm_backward_recursion}).
\end{itemize}
Since in that case the JT is in fact reduced to a simple sequence, messages in the forward and backward directions can be computed independently. This is however not true in the general case where a more subtle recursion algorithm is needed.
\end{example}

\begin{proposition}[inward-outward algorithm]
If we choose a root $r \in \mathcal{I}$ for the JT, we call \emph{inward message} any message orientated from leaves to the root, and \emph{outward message} any message in the opposite direction. We define on $i\in \mathcal{I}$ two recursive function:
\begin{itemize}
\item ${\tt inward}(i)$: {\bf for all} $j$ offspring of $i$  {\bf do} call ${\tt inward}(j)$, and compute $M_{j\rightarrow i}$;
\item ${\tt outward}(i)$: {\bf for all} $j$ offspring of $i$  {\bf do} compute $M_{i\rightarrow j}$, and call ${\tt outward}(j)$.
\end{itemize}
Then all inward messages can be computed by calling ${\tt inward}(r)$, and then, the remaining outward messages by calling ${\tt outward}(r)$.
\end{proposition}
\begin{proof}
See classical textbooks \citep{cowell99,jensen07,koller09} for a detailed proof.
\end{proof}

One should note that if the inward recursion only involve inward messages, the outward recursion involves both inward and outward messages. This means that unlike with the forward-backward recursion in HMM, the two recursions cannot be done independently. Another interesting remark is that thanks to Eq.~(\ref{eq:marginal2}),  the recursion ${\tt inward}(r)$ is sufficient to obtain $\mathbb{P}(X_{C_{r}},\mathcal{E})$ and hence also $\mathbb{P}(\mathcal{E})$.

\begin{example}
If we now come back to the precipitation HMM and if we root the JT in $r=n$, then ${\tt inward}(n)$ perform the standard forward recursion, and ${\tt outward}(n)$ perform the backward one. However, other rooting are possible. For example if we choose $r=i\in\mathcal{I}$ with $i\neq 1$ and $i\neq n$, then ${\tt inward}(i)$ allows to compute $\mathbb{P}(S_{i-1},S_{i},Y_{1:n}=y_{1:n})$ directly, the inward messages involved in the process being a mixture of forward and backward messages.
\end{example}

\begin{example}
For the pedigree BNT with root $r=1$, the inward recursion is:
\begin{itemize}
\item $M_{7\rightarrow 6}(X_3,X_5)=\sum_{X_8} K_8(X_3,X_5,X_8)$;
\item $M_{6\rightarrow 4}(X_3,X_7)=\sum_{X_5} K_5(X_5)K_7(X_3,X_5,X_7)M_{7\rightarrow 6}(X_3,X_5)$;
\item $M_{5\rightarrow 4}(X_7,X_9)=\sum_{X_{10}} K_{10}(X_7,X_9,X_{10})$;
\item $M_{4\rightarrow 2}(X_3,X_9)=\sum_{X_7} M_{5\rightarrow 4}(X_7,X_9)M_{6\rightarrow 4}(X_3,X_7)$;
\item $M_{3\rightarrow 2}(X_4,X_9)=\sum_{X_6} K_6(X_6)K_9(X_4,X_6,X_9)$;
\item $M_{2\rightarrow 1}(X_3,X_4)=\sum_{X_9} M_{3\rightarrow 2}(X_4,X_9)M_{4\rightarrow 2}(X_3,X_9)$.
\end{itemize}
The outward recursion is (inward messages are underlined):
\begin{itemize}
\item $M_{1\rightarrow 2}(X_3,X_4)=\sum_{X_1,X_2}  K_1(X_1)K_2(X_2)K_3(X_1,X_2,X_3)K_4(X_1,X_2,X_4)$;
\item $M_{2\rightarrow 3}(X_4,X_9)=\sum_{X_3} M_{1\rightarrow 2}(X_3,X_4)\underline{M_{4\rightarrow 2}(X_3,X_9)}$;
\item $M_{2\rightarrow 4}(X_3,X_9)=\sum_{X_4} M_{1\rightarrow 2}(X_3,X_4)\underline{M_{3\rightarrow 2}(X_4,X_9)}$;
\item $M_{4\rightarrow 5}(X_7,X_9)=\sum_{X_{3}} M_{2\rightarrow 4}(X_3,X_9) \underline{M_{6\rightarrow 4}(X_3,X_7)}$;
\item $M_{4\rightarrow 6}(X_3,X_7)=\sum_{X_9} M_{2\rightarrow 4}(X_3,X_9) \underline{M_{5\rightarrow 4}(X_7,X_9)}$;
\item $M_{6\rightarrow 7}(X_3,X_5)=\sum_{X_7} K_5(X_5)K_7(X_3,X_5,X_7)M_{4\rightarrow 6}(X_3,X_7)$.
\end{itemize}

The results of these recursions are given in Table~\ref{table:ped_messages}.
\end{example}

\begin{table}
\caption{Messages of the pedigree BNT. First part of the table corresponds to the inward messages computed using Cluster~$1$ as root. The second part of the table corresponds to the outward messages.}\label{table:ped_messages}
\vspace{-2em}
$$
\small
{\setlength{\arraycolsep}{4pt}
\begin{array}{cccccccccc}
\hline
X_i,X_j & {\tt dd},{\tt dd} & {\tt dd},{\tt dD} & {\tt dd},{\tt DD} & {\tt dD},{\tt dd} & {\tt dD},{\tt dD} &  {\tt dD},{\tt DD} & {\tt DD},{\tt dd} & {\tt DD},{\tt dD} & {\tt DD},{\tt DD} \\
\hline
M_{7\rightarrow 6}(X_3,X_5) &  0.0000 & 0.0000  &  0.0000 & 0.0000 & 0.2500 & 0.5000 & 0.0000 & 0.5000 & 1.0000\\
M_{6\rightarrow 4}(X_3,X_7) &  0.0000 & 0.0000  &  0.0000 & 0.0200 & 0.0500 & 0.0000 & 0.0000 & 0.0800 & 0.0000\\
M_{5\rightarrow 4}(X_7,X_9) &  0.0000 & 0.0000  &  0.0000 & 0.0000 & 0.2500 & 0.5000 & 0.0000 & 0.5000 & 1.0000\\
M_{4\rightarrow 2}(X_3,X_9) &  0.0000 & 0.0000  &  0.0000 & 0.0000 & 0.0250 & 0.0500 & 0.0000 & 0.0400 & 0.0800\\
M_{3\rightarrow 2}(X_4,X_9) &  0.8000 & 0.2000  &  0.0000 & 0.4000 & 0.5000 & 0.1000 & 0.0000 & 0.8000 & 0.2000\\
M_{2\rightarrow 1}(X_3,X_4) &  0.0000 & 0.0000  &  0.0000 & 0.0025 & 0.0088 & 0.0150 & 0.0040 & 0.0140 & 0.0240\\
\hline
1000 \times M_{1\rightarrow 2}(X_3,X_4)  &  0.0000 &0.0000 &0.0000 &0.0000 &0.0000 &3.2000 &0.0000 &0.0000 &4.8000\\
1000 \times M_{2\rightarrow 3}(X_4,X_9)  &  0.0000 &0.0000 &0.0000 &0.0000 &0.0000 &0.0000 &0.0000 &0.1360 &0.2720\\
1000 \times M_{2\rightarrow 4}(X_3,X_9)  &  0.0000 &0.0000 &0.0000 &0.0000 &2.5600 &0.6400 &0.0000 &3.8400 &0.9600\\
1000 \times M_{4\rightarrow 5}(X_7,X_9)  &  0.0000 &0.0512 &0.0128 &0.0000 &0.4352 &0.1088 &0.0000 &0.0000 &0.0000\\
1000 \times M_{4\rightarrow 6}(X_3,X_7)  &  0.0000 &0.0000 &0.0000 &0.0000 &0.9600 &1.9200 &0.0000 &1.4400 &2.8800\\
1000 \times M_{6\rightarrow 7}(X_3,X_5)  &  0.0000 &0.0000 &0.0000 &0.3072 &0.1536 &0.0192 &0.9216 &0.2304 &0.0000\\
\hline
\end{array}
}
$$

\end{table}

\subsection{Sampling}

\begin{corollary}
For all edge $j-k$ of the JT, for all $X_{C_{j}} \in \mathcal{D}_{C_{j}}$ we have:
	\begin{equation}
	\mathbb{P}\left(X_{C_j} | X_{S_{j,k}},\mathcal{E}\right)
	=\frac{\Phi_j\left( X_{C_j} \right)
\prod_{i \in \text{n}(j),i \neq k} M_{i\rightarrow j}\left(X_{S_{i,j}}\right)}
	{M_{j\rightarrow k}\left(X_{S_{j,k}} \right)}
	\end{equation}
\end{corollary}
\begin{proof}
Immediate by dividing Eq.~(\ref{eq:marginal2}) by Eq.~(\ref{eq:marginal1}).
\end{proof}

\begin{example}
In the particular case of the precipitation HMM, we get:
\begin{itemize}
\item for all $i=1\ldots n-2$, 
$$
\mathbb{P}(S_{i+1}|S_{i},Y_{1:n}=y_{1:n})=
\frac{\mathbb{P}(S_{i+1}|S_{i})\mathbb{P}(Y_{i+1}=y_{i+1}|S_{i+1})M_{i \rightarrow i-1}\left(S_{i-1}\right)}
{M_{i+1 \rightarrow i}\left(S_{i}\right)}$$
 which is exactly Eq.~(\ref{eq:hmm_forward_sampling});
\item for all $i=2\ldots n-1$, 
$$
\mathbb{P}(S_{i-1}|S_i,Y_{1:n}=y_{1:n})=
\frac{\mathbb{P}(S_i|S_{i-1})\mathbb{P}(Y_i=y_i|S_i)M_{i-1 \rightarrow i}\left(S_{i-1}\right)}
{M_{i \rightarrow i+1}\left(S_{i}\right)}$$
 which is exactly Eq.~(\ref{eq:hmm_backward_sampling}).
\end{itemize}
Both formulas allows to sample from $\mathbb{P}(S_{1:n} | Y_{1:n}=y_{1:n})$ sequentially (either in the forward or backward direction). Like for the recursions in previous section, this is due to the particular structure of the JT (a sequence) and a more subtle sampling algorithm is necessary in general.
\end{example}

\begin{proposition}[sampling]
For any root $r \in \mathcal{U}$, a sample from $\mathbb{P}(X_\mathcal{U} | \mathcal{E})$ is recursively obtained by calling  ${\tt inward}(r)$ and then ${\tt sample}(r)$ with
\begin{itemize}
\item ${\tt sample}(i)$: draw $\mathbb{P}(X_{C_i} | X_{S_{i,\text{pa}(i)}}, \mathcal{E})$ and {\bf for all} $j$ offspring of $i$  call ${\tt sample}(j)$
\end{itemize}
where $\text{pa}(i)$ denotes the parent of $i$ in the rooted JT and $S_{r,\text{pa}(r)}=\emptyset$ by convention.
\end{proposition}
\begin{proof}
The proof is the same than for the inward recursion.
\end{proof}

\begin{example}
For the pedigree BNT, sampling from $\mathbb{P}(X_{1:10} | \mathcal{E})$ is achieved through:
\begin{itemize}
\item sample $(X_1,X_2,X_3,X_4)$ from $\mathbb{P}(X_1,X_2,X_3,X_4|\mathcal{E})={\displaystyle\frac{\mathbb{P}(X_1,X_2,X_3,X_4,\mathcal{E})}{\mathbb{P}(\mathcal{E})}}$;
\item sample $X_9$ from $\mathbb{P}(X_9|X_3,X_4,\mathcal{E})={\displaystyle\frac{M_{3\rightarrow 2}(X_4,X_9)M_{4\rightarrow 2}(X_3,X_9)}{M_{2\rightarrow 1}(X_3,X_4)}}$;
\item sample $X_6$ from $\mathbb{P}(X_6|X_4,X_9,\mathcal{E})={\displaystyle\frac{K_6(X_6)K_9(X_4,X_6,X_9)}{M_{3\rightarrow 2}(X_4,X_9)}}$;
\item sample $X_7$ from $\mathbb{P}(X_7|X_3,X_9,\mathcal{E})={\displaystyle\frac{M_{5\rightarrow 4}(X_7,X_9)M_{6\rightarrow 4}(X_3,X_7)}{M_{4\rightarrow 2}(X_3,X_9)}}$;
\item sample $X_{10}$ from $\mathbb{P}(X_{10}|X_7,X_9,\mathcal{E})={\displaystyle\frac{K_{10}(X_7,X_9,X_{10})}{M_{5\rightarrow 4}(X_7,X_9)}}$;
\item sample $X_{5}$ from $\mathbb{P}(X_{5}|X_3,X_7,\mathcal{E})={\displaystyle\frac{K_5(X_5)K_7(X_3,X_5,X_7)M_{7\rightarrow 6}(X_3,X_5)}{M_{6\rightarrow 4}(X_3,X_7)}}$;
\item sample $X_{8}$ from $\mathbb{P}(X_{8}|X_3,X_5,\mathcal{E})={\displaystyle\frac{K_8(X_3,X_5,X_8)}{M_{7\rightarrow 6}(X_3,X_5)}}$;
\end{itemize}
We can see on Table~\ref{tab:ped_samples} five samples drawn from $\mathbb{P}(X_{1:10} | \mathcal{E})$ using these conditional probabilities.
\end{example}

One should note that it also possible to sample from $\mathbb{P}(X_\mathcal{V} | \mathcal{E})$ for any $\mathcal{V} \subset \mathcal{U}$ in a slightly more efficient way by restraining the sampling recursion to a subtree of the JT.

\begin{table}[t]
\caption{Five samples drawn from $\mathbb{P}(X_{1:10} | \mathcal{E})$. The marginal posterior distribution of each variable is also given.}
\label{tab:ped_samples}

\vspace{-1em}
$$
{
\setlength\arraycolsep{5pt}
\begin{array}{ccccccccccc}
\hline
\text{variable} & X_1 & X_2 & X_3 & X_4 & X_5 & X_6 & X_7 & X_8 & X_9 & X_{10} \\
\hline 
\text{sample 1} & {\tt dD} &  {\tt DD} &  {\tt dD} &  {\tt DD} &  {\tt DD} &  {\tt dD} &  {\tt dD} &  {\tt DD} &  {\tt DD} &  {\tt DD} \\ 
\text{sample 2} & {\tt dD} &  {\tt DD} &  {\tt DD} &  {\tt DD} &  {\tt dD} &  {\tt dD} &  {\tt dD} &  {\tt DD} &  {\tt DD} &  {\tt DD} \\ 
\text{sample 3} & {\tt DD} &  {\tt DD} &  {\tt DD} &  {\tt DD} &  {\tt dD} &  {\tt dd} &  {\tt dD} &  {\tt DD} &  {\tt dD} &  {\tt DD} \\ 
\text{sample 4} & {\tt dD} &  {\tt DD} &  {\tt dD} &  {\tt DD} &  {\tt DD} &  {\tt dd} &  {\tt dD} &  {\tt DD} &  {\tt dD} &  {\tt DD} \\ 
\text{sample 5} & {\tt dD} &  {\tt DD} &  {\tt dD} &  {\tt DD} &  {\tt dD} &  {\tt dd} &  {\tt dD} &  {\tt DD} &  {\tt dD} &  {\tt DD} \\
\hline
\mathbb{P}(X_\cdot = {\tt dd} | \mathcal{E}) & 0.00 & 0.00 & 0.00 & 0.00 & 0.00 & 0.53 & 0.00 & 0.00 & 0.00 &  0.00 \\
\mathbb{P}(X_\cdot = {\tt dD} | \mathcal{E}) & 0.76 & 0.00 & 0.29 & 0.00 & 0.94 & 0.40 & 1.00 & 0.00 & 0.67 &  0.00\\
\mathbb{P}(X_\cdot = {\tt DD} | \mathcal{E}) & 0.24 & 1.00 & 0.71 & 1.00 & 0.06 & 0.07 & 0.00 & 1.00 & 0.33 &  1.00 \\
\hline
\end{array}
}
$$
\end{table}

\section{Discussion}

We have introduced here with Eq.~(\ref{eq:messagedef0}) an explicit definition of messages in BNTs. To the best of our knowledge, this surprisingly seems to be the first time. Indeed, when looking either in the founding papers and textbooks where exact BP was initially developed \citep{pearl86,pearl88,lauritzen88,shafer90,jensen90bis,jensen90}, or in the most recent work on the subject \citep{jensen07,koller09,tarlow2010,caetano11}, messages are always defined implicitly through the recursive formula of Eq.~(\ref{eq:recursion}). This might be due to the fact that the popular approximated BP algorithms (ex: loopy BP) are all based on similar recursive formulas.

However, the explicit message-centered approach that we suggest here has several advantages over the classical approach of exact BP. Firstly, it follows the sketch of the theory of inference in HMMs allowing to introduce BNTs as a natural extension of these well-known models from definitions to proofs, with obvious pedagogical benefits. Secondly, it provides a compact and straightforward proof of all exact BP results (the only steps which require some work are Lemma~\ref{lemma:messages}, Proposition~\ref{prop:marginal1}, and Proposition~\ref{prop:marginal2}). Finally, it extends a step further the parallelism pointed out by \citet{smyth97} between Markov sequence related models (Markov chains, HMMs, Markov tree) and BNTs therefore opening new exciting possibilities for those who work with HMMs models and variants without having to refer to the general theory of exact BP in BNTs to prove the resulting formulas.

For example, suppose we consider $X_{1:n}$ an homogeneous Markov chain with starting distribution $\mu$ and transition matrix $\pi$, and would like to sample from $\mathbb{P}(X_{1:n} | X_1=X_n)$. By introducing an appropriate BNT (left to the reader), we can easily establish that $\mathbb{P}(X_1 | X_1=X_n) \propto \mu(X_1)\pi^{n-1}(X_1,X_1)$ and that
$\mathbb{P}(X_{i} | X_{i-1},X_1,X_1=X_n) \propto {\pi(X_{i-1},X_{i})\pi^{n-i}(X_i,X_1)}$ for all $i=2\ldots n-1$. Of course, this result can be obtained directly without introducing any BNT, but our message-centered approach provides without effort a complete sketch of the proof. This might prove itself very useful when working with sophisticated extension of Markov sequence related models (ex: HMMs with partially observed hidden states, complex dependencies, or multiple observations; evolutionary processes through Markov trees including loops, etc.).

For further work, it would be interesting to extend our approach to more general propagation than the sum-product one we consider here. For example, max-product propagation can be easily considered by replacing sums by maximums in Eq.~(\ref{eq:messagedef}), thus giving the following max-message definition:
\begin{equation}
	M_{i \rightarrow j}^\text{max} \left(X_{S_{i,j}}\right)\stackrel{\text{def}}{=}
		\mathbf{1}_{\mathcal{E}_{L_{i\rightarrow j}}}%
		\max_{X_{{V_{i\rightarrow j}}} \in \mathcal{X}_{{V_{i\rightarrow j}}}}
	\mathbb{P}\left(
	X_{L_{i\rightarrow j}},
	X_{{V_{i\rightarrow j}}}
	\left| X_{L_{j\rightarrow i}} \right.
	\right)
\end{equation}
from which all max-product propagation results can be easily derived.

\appendix

\section*{Appendix}

\section{R source code for the precipitation HMM}

\begin{verbatim}
# generates the data
pi=matrix(c(0.7,0.3,0.1,0.9),ncol=2,byrow=T);
n=100;
s=numeric(n);
s[1]=2;
for (i in 2:100) s[i]=which(rmultinom(1,size=1,prob=pi[s[i-1],])==1);
lambda=c(3.0,0.5);
x=rpois(n,lambda=lambda[s]);
plot(x);
index=1:n;
points(index[s==1],x[s==1],col="blue");
points(index[s==2],x[s==2],col="red");

# forward and backward recursions
e=rbind(dpois(x,lambda[1]),dpois(x,lambda[2]));
F=0*e; B=0*e;
F[2,1]=e[2,1];
for (i in 2:n) F[,i]=t(F[,i-1]%*%pi)*e[,i];
B[,n]=1;
for (i in seq(n-1,1,by=-1)) B[,i]=pi%*%(e[,i+1]*B[,i+1]);

# marginal distribution
marginal=B*F/sum(B[,1]*F[,1]);
plot(marginal[1,],t='l',col="blue",lwd=2);
points(marginal[2,],t='l',col="red",lwd=2);
points(s==1,col="blue");
points(s==2,col="red");

# sampling from P(S|X)
sample=NULL;
for (iter in 1:5) {
  ss=numeric(n);
  ss[1]=2;
  for (i in 2:100) ss[i]=which(rmultinom(1,size=1,
  	prob=pi[ss[i-1],]/B[ss[i-1],i-1]*e[,i]*B[,i])==1);
  sample=rbind(sample,ss);
}
\end{verbatim}

\section{R source code for the pedigree BNT}

\begin{verbatim}
# define the model
p=0.2
Pf=c((1-p)^2,2*p*(1-p),p^2);
Pnf=matrix(rep(NA,27),nrow=3);
Pnf[1,]=c(1,0.5,0,0.5,0.25,0,0,0,0);
Pnf[2,]=c(0,0.5,1,0.5,0.5,0.5,1,0.5,0);
Pnf[3,]=c(0,0,0,0,0.25,0.5,0,0.5,1);
pair=function(X1,X2) 3*(X1-1)+X2;
K1=function(X1) Pf[X1];
K2=function(X2) (X2==3)*Pf[X2];
K3=function(X1,X2,X3) Pnf[X3,pair(X1,X2)];
K4=function(X1,X2,X4) (X4==3)*Pnf[X4,pair(X1,X2)];
K5=function(X5) Pf[X5];
K6=function(X6) Pf[X6];
K7=function(X3,X5,X7) (X7!=3)*Pnf[X7,pair(X3,X5)];
K8=function(X3,X5,X8) (X8==3)*Pnf[X8,pair(X3,X5)];
K9=function(X4,X6,X9) Pnf[X9,pair(X4,X6)];
K10=function(X7,X9,X10) (X10==3)*Pnf[X10,pair(X7,X9)];

# inward
M76=rep(0,9);
for (X3 in 1:3) for (X5 in 1:3) for (X8 in 1:3)
  M76[pair(X3,X5)]=M76[pair(X3,X5)]+K8(X3,X5,X8);
M64=rep(0,9);
for (X3 in 1:3) for (X7 in 1:3) for (X5 in 1:3)
  M64[pair(X3,X7)]=M64[pair(X3,X7)]+K5(X5)*K7(X3,X5,X7)*M76[pair(X3,X5)];
M54=rep(0,9);
for (X7 in 1:3) for (X9 in 1:3) for (X10 in 1:3)
  M54[pair(X7,X9)]=M54[pair(X7,X9)]+K10(X7,X9,X10);
M42=rep(0,9);
for (X3 in 1:3) for (X9 in 1:3) for (X7 in 1:3)
  M42[pair(X3,X9)]=M42[pair(X3,X9)]+M54[pair(X7,X9)]*M64[pair(X3,X7)];
M32=rep(0,9);
for (X4 in 1:3) for (X9 in 1:3) for (X6 in 1:3)
  M32[pair(X4,X9)]=M32[pair(X4,X9)]+K6(X6)*K9(X4,X6,X9);
M21=rep(0,9);
for (X3 in 1:3) for (X4 in 1:3) for (X9 in 1:3)
  M21[pair(X3,X4)]=M21[pair(X3,X4)]+M32[pair(X4,X9)]*M42[pair(X3,X9)];

# outward
M12=rep(0,9);
for (X3 in 1:3) for (X4 in 1:3) for (X1 in 1:3) for (X2 in 1:3)
  M12[pair(X3,X4)]=M12[pair(X3,X4)]+K1(X1)*K2(X2)*K3(X1,X2,X3)*K4(X1,X2,X4);
M23=rep(0,9);
for (X4 in 1:3) for (X9 in 1:3) for (X3 in 1:3)
  M23[pair(X4,X9)]=M23[pair(X4,X9)]+M12[pair(X3,X4)]*M42[pair(X3,X9)];
M24=rep(0,9);
for (X3 in 1:3) for (X9 in 1:3) for (X4 in 1:3)
  M24[pair(X3,X9)]=M24[pair(X3,X9)]+M12[pair(X3,X4)]*M32[pair(X4,X9)];
M45=rep(0,9);
for (X7 in 1:3) for (X9 in 1:3) for (X3 in 1:3)
  M45[pair(X7,X9)]=M45[pair(X7,X9)]+M24[pair(X3,X9)]*M64[pair(X3,X7)];
M46=rep(0,9);
for (X3 in 1:3) for (X7 in 1:3) for (X9 in 1:3)
  M46[pair(X3,X7)]=M46[pair(X3,X7)]+M24[pair(X3,X9)]*M54[pair(X7,X9)];
M67=rep(0,9);
for (X3 in 1:3) for (X5 in 1:3) for (X7 in 1:3)
  M67[pair(X3,X5)]=M67[pair(X3,X5)]+K5(X5)*K7(X3,X5,X7)*M46[pair(X3,X7)];
pevidence=sum(M12*M21);
pevidence=sum(M67*M76);
print(rbind(M76,M64,M54,M42,M32,M21),digits=12);
print(rbind(M12,M23,M24,M45,M46,M67)*1000,digits=12);

# marginal distributions
P1=rep(0,3);
for (X1 in 1:3) for (X2 in 1:3) for (X3 in 1:3) for (X4 in 1:3)
 P1[X1]=P1[X1]+K1(X1)*K2(X2)*K3(X1,X2,X3)*K4(X1,X2,X4)*M21[pair(X3,X4)];
P2=rep(0,3);
for (X1 in 1:3) for (X2 in 1:3) for (X3 in 1:3) for (X4 in 1:3)
 P2[X2]=P2[X2]+K1(X1)*K2(X2)*K3(X1,X2,X3)*K4(X1,X2,X4)*M21[pair(X3,X4)];
P3=rep(0,3);
for (X1 in 1:3) for (X2 in 1:3) for (X3 in 1:3) for (X4 in 1:3)
 P3[X3]=P3[X3]+K1(X1)*K2(X2)*K3(X1,X2,X3)*K4(X1,X2,X4)*M21[pair(X3,X4)];
P4=rep(0,3);
for (X1 in 1:3) for (X2 in 1:3) for (X3 in 1:3) for (X4 in 1:3)
 P4[X4]=P4[X4]+K1(X1)*K2(X2)*K3(X1,X2,X3)*K4(X1,X2,X4)*M21[pair(X3,X4)];
P5=rep(0,3);
for (X3 in 1:3) for (X5 in 1:3)
  P5[X5]=P5[X5]+M67[pair(X3,X5)]*M76[pair(X3,X5)];
P6=rep(0,3);
for (X4 in 1:3) for (X6 in 1:3) for (X9 in 1:3)
  P6[X6]=P6[X6]+K6(X6)*K9(X4,X6,X9)*M23[pair(X4,X9)];
P7=rep(0,3);
for (X7 in 1:3) for (X9 in 1:3)
  P7[X7]=P7[X7]+M45[pair(X7,X9)]*M54[pair(X7,X9)];
P8=rep(0,3);
for (X3 in 1:3) for (X5 in 1:3) for (X8 in 1:3)
 P8[X8]=P8[X8]+K8(X3,X5,X8)*M67[pair(X3,X5)];
P9=rep(0,3);
for (X4 in 1:3) for (X6 in 1:3) for (X9 in 1:3)
  P9[X9]=P9[X9]+K6(X6)*K9(X4,X6,X9)*M23[pair(X4,X9)];
P10=rep(0,3);
for (X7 in 1:3) for (X9 in 1:3) for (X10 in 1:3)
 P10[X10]=P10[X10]+K10(X7,X9,X10)*M45[pair(X7,X9)];

# sampling
P13=rep(0,9);
for (X1 in 1:3) for (X2 in 1:3) for (X3 in 1:3) for (X4 in 1:3)
 P13[pair(X1,X3)]=P13[pair(X1,X3)]+K1(X1)*K2(X2)*
 	K3(X1,X2,X3)*K4(X1,X2,X4)*M21[pair(X3,X4)];
sample=matrix(rep(NA,5*10),nrow=5);
for (iter in 1:5) {
  sample[iter,2]=3;
  sample[iter,4]=3;
  sample[iter,8]=3;
  sample[iter,10]=3;
  aux=which(rmultinom(1, size=1, prob=P13/pevidence)==1);
  sample[iter,1]=floor((aux-1)/3)+1;
  sample[iter,3]=aux-3*floor((aux-1)/3);
  CP9=rep(NA,3);
  for (X9 in 1:3) {
    CP9[X9]=M32[pair(sample[iter,4],X9)]*M42[pair(sample[iter,3],X9)]/
    	M21[pair(sample[iter,3],sample[iter,4])];
  }
  sample[iter,9]=which(rmultinom(1, size=1, prob=CP9)==1);
  CP6=rep(NA,3);
  for (X6 in 1:3) {
    CP6[X6]=K6(X6)*K9(sample[iter,4],X6,sample[iter,9])/
    	M32[pair(sample[iter,4],sample[iter,9])];
  }
  sample[iter,6]=which(rmultinom(1, size=1, prob=CP6)==1);
  CP7=rep(NA,3);
  for (X7 in 1:3) {
    CP7[X7]=M54[pair(X7,sample[iter,9])]*M64[pair(sample[iter,3],X7)]/
    	M42[pair(sample[iter,3],sample[iter,9])];
  }
  sample[iter,7]=which(rmultinom(1, size=1, prob=CP7)==1);
  CP5=rep(NA,3);
  for (X5 in 1:3) {
    CP5[X5]=K5(X5)*K7(sample[iter,3],X5,sample[iter,7])*
    	M76[pair(sample[iter,3],X5)]/M64[pair(sample[iter,3],sample[iter,7])];
  }
  sample[iter,5]=which(rmultinom(1, size=1, prob=CP5)==1);
}
\end{verbatim}

\bibliographystyle{plainnat}

\bibliography{message}

\end{document}